\documentclass[11pt,reqno,tbtags]{amsart}
\usepackage[]{amsmath,amssymb,amsfonts,latexsym,amsthm,enumerate,paralist}
\usepackage{amssymb,bbm,mathrsfs}
\usepackage{enumerate,graphicx}
\usepackage[numeric,initials,nobysame]{amsrefs}
\usepackage[marginratio=1:1,height=584pt,width=468pt,tmargin=117pt]{geometry}

\numberwithin{equation}{section}

\newtheorem{conjecture}{Conjecture}
\newtheorem{theorem}{Theorem}[section]
\newtheorem*{theorem*}{Theorem}
\newtheorem{lemma}[theorem]{Lemma}
\newtheorem{claim}[theorem]{Claim}
\newtheorem{proposition}[theorem]{Proposition}
\newtheorem{observation}[theorem]{Observation}

\theoremstyle{definition}{

\newtheorem*{definition*}{Definition}

\newtheorem*{remark*}{Remark}
}

\renewcommand{\P}{\mathbb{P}}

\newcommand{\Bin}{\operatorname{Bin}}

\renewcommand{\epsilon}{\varepsilon}
\renewcommand{\phi}{\varphi}

\newcommand{\cG}{\mathcal{G}}

\newcommand{\comb}{\text{\sffamily\slshape Comb}}

\newcommand{\remove}[1]{}

\def\aas{w.h.p.}
\def\Aas{W.h.p.}

\newcommand{\beq}[1]{\begin{equation}\label{#1}}
\newcommand{\enq}[0]{\end{equation}}

\newcommand{\bn}[0]{\bigskip\noindent}
\newcommand{\mn}[0]{\medskip\noindent}
\newcommand{\nin}[0]{\noindent}

\newcommand{\sub}[0]{\subseteq}
\newcommand{\sm}[0]{\setminus}
\renewcommand{\dots}[0]{,\ldots,}

\newcommand{\ov}[0]{\overline}

\newcommand{\ra}[0]{\rightarrow}

\newcommand{\FF}[0]{{\bf F}}

\newcommand{\rrr}[0]{r}
\newcommand{\DDD}[0]{D}

\newcommand{\0}[0]{\emptyset}

\newcommand{\C}[2]{{{#1}\choose{{#2}}}}
\newcommand{\Cc}[0]{\tbinom}
\newcommand{\ga}[0]{\alpha }
\newcommand{\gb}[0]{\beta }
\newcommand{\gc}[0]{\gamma }

\newcommand{\gD}[0]{\Delta }
\newcommand{\gG}[0]{\Gamma }

\newcommand{\gO}[0]{\Omega}

\newcommand{\gz}[0]{\zeta}
\newcommand{\eps}[0]{\varepsilon }
\newcommand{\vt}[0]{\vartheta}

\newcommand{\sugg}[1]{}
\newcommand{\sugr}[1]{}

\date{}

\begin{document}
\title{The threshold for combs in random graphs}

\author{Jeff Kahn}
\address{Jeff Kahn\hfill\break
Department of Mathematics\\
Rutgers\\
Piscataway, NJ 08854, USA.}
\email{jkahn@math.rutgers.edu}
\urladdr{}
\thanks{J.\ Kahn is supported by NSF grant DMS0701175.}
\urladdr{}

\author{Eyal Lubetzky}
\address{Eyal Lubetzky\hfill\break
Microsoft Research\\
One Microsoft Way\\
Redmond, WA 98052, USA.}
\email{eyal@microsoft.com}
\urladdr{}

\author{Nicholas Wormald}
\address{Nicholas Wormald\hfill\break
School of Mathematical Sciences\\
Monash University\\
Clayton, Victoria 3800, Australia.}
\email{nick.wormald@monash.edu}
\thanks{N.\ Wormald was supported by the Canada Research Chairs Program and NSERC during this research.}
\urladdr{}

\begin{abstract}
For $k\mid n$ let $\comb_{n,k}$ denote the tree consisting
of an $(n/k)$-vertex path with disjoint $k$-vertex paths beginning at each of its vertices.
An old conjecture says that for any $k=k(n)$ the threshold for the random graph $\cG(n,p)$
to contain $\comb_{n,k}$ is at $p\asymp \frac{\log n}n$.
Here we verify this for $k \leq C\log n$ with any fixed $C>0$.
In a companion paper, using very different methods, we treat the complementary range, proving the conjecture for $k\geq \kappa_0 \log n$ (with $\kappa_0\approx 4.82$).
\end{abstract}

\maketitle

\section{Introduction}


Write $G=\cG(n,p)$ for the usual random graph on
$V:=[n]:=\{1\dots n\}$, in which edges are present independently,
each with probability $p$.
We are interested in understanding when (i.e.\ for what $p$) $G$ is
likely to contain (a copy of) a fixed $n$-vertex tree $T$.

(Formally we may define the ``threshold" for containing $T$ to be
that (unique) $p$ for which the probability that $G$ contains $T$ is 1/2.
To stay closer to the usual threshold language of
\cite{ER}, or e.g. \cite{JLR}, we would need to work with a sequence $\{T_n\}$;
but in any case, we will not make much use of the formal definition.)

Specifically we are interested in the following conjecture.
\begin{conjecture}\label{TConj}
For each fixed $\gD$ there is a $C$ such that
if $T$ is any
$n$-vertex tree of maximum degree at most $\gD$,
then $\cG(n,C\frac{\log n}n)$ \aas\ contains $T$.
\end{conjecture}
\nin
(As usual ``\aas" means with probability tending to 1 as $n\ra \infty$.)

Of course for $p < \frac{\log n}n$
(we use $\log$ for $\ln$), $G$ is
likely to contain isolated vertices, so Conjecture~\ref{TConj} says
that the threshold for containing any bounded degree $T$ is
$\Theta(\frac{\log n}n)$.
This is known when $T$ is a Hamiltonian path
\cites{KS,Boll}, and easy when $T$ has $\gO(n)$ leaves
(see \cites{AKS,Kriv}).
It has also been proved for ``almost all" trees,
even without the maximum degree requirement \cite{BW}.
More recently \cite{HKS}, it has been shown to hold with
$C=1+\eps$ if $T$ has $\gO(n)$ leaves or contains a path of length $\gO(n)$ consisting of vertices of
degree 2.
The best general progress to date is \cite{Kriv},
which proves that $p\geq n^{-1+o(1)}$ suffices for all bounded degree trees,
and also considers larger degrees; see this reference for some further discussion.

Conjecture~\ref{TConj}
was proposed by the first author about twenty years ago
(though stated in print only in \cite{KK}, in which see also
the far more general \cite{KK}*{Conjecture 1}),
but, being a natural guess,
is perhaps best considered folklore.
At that early date it was also suggested that some insight might be
gained by considering the case where,
for some $k\mid n$, $T$ is the tree --- here denoted $\comb_{n,k}$ --- consisting
of an $(n/k)$-vertex path $P$
together with disjoint $k$-vertex paths beginning at the vertices of $P$.
Such trees, which have sometimes been called ``combs,"
may be thought of as lying somewhere between the settled cases of
Conjecture~\ref{TConj} mentioned above.

Though we have not much non-verbal evidence,
this suggestion does seem to have received quite a bit of attention,
but, absent
any serious progress, seems not to have produced anything in print.
Here and in the companion paper~\cite{KLW2} we establish Conjecture~\ref{TConj}
for combs.

\begin{theorem}\label{T1}
There exists some fixed $C$ such that
for every $n$ and $k\mid n$, the random graph $\cG(n,C\frac{\log n}n)$
\aas\ contains a copy of $\comb_{n,k}$.
\end{theorem}
\nin
While this does not so far seem to be leading to a proof
of Conjecture~\ref{TConj}, it is plausible that our methods at least extend to any (bounded-degree) tree with $o(\sqrt{n})$ leaves.


The proof of Theorem~\ref{T1}
requires two entirely different arguments, depending on
whether $k$ is large (at least about $\log n$) or small.
Here we treat small $k$.

\begin{theorem}\label{T3}
For each $\DDD$
there is a $K$
for which the following holds.
If $k<\DDD\log n$ divides $n$, and $v_1,\ldots,v_m$ are $m=n/k$ given (distinct) vertices, then
$\cG(n, K \frac{\log n}n)$ \aas\  contains $m$ disjoint $k$-vertex paths rooted at the $v_i$'s.
\end{theorem}

\nin
This is proved in Section~\ref{Proof}.
For the easy derivation of
Theorem~\ref{T1} (for small $k$), we may take
$G=G'\cup G''$, where $G'$ and $G''$ are independent copies of,
respectively,
$\cG(n, d/n)$ for a suitable constant $d$, and $\cG(n,p)$.
(So the $p$ in Theorem~\ref{T1} will be slightly larger
than the one in Theorem~\ref{T3}.)
Then $G'$ \aas\  contains a path $v_1\dots v_m$
(assuming, as we may, that $k>1$; see, e.g., \cite{RG}*{Chap. 8}),
which, according to Theorem~\ref{T3},
 we can (\aas) extend to a copy of $\comb_{n,k}$ using $G''$.

\section{Proof of Theorem~\ref{T3}}\label{Proof}

For a graph $H$ on $V$ and disjoint $A,B\sub V$,
we use the notation
$\nabla_H(A,B)=\{xy\in E(H):x\in A,y\in B\}$,
omitting the subscript when $H$ is the complete graph $K_V$.
As above, we write $G$ for $\cG(n,p)$.
Following common practice, we will sometimes pretend large numbers
are integers to avoid cluttering the discussion
with irrelevant floor and ceiling symbols.


Since Conjecture~\ref{TConj} is known to hold
when $T$ has $\gO(n)$ leaves, we may assume $k$ is
at least any given constant.
Though not really necessary, this
will save us a little trouble
in some places.
\remove{
and there
seems no point in taking extra pains to cover a range
in which is more or less trivial.}
Specifically we assume (as we may) that $\DDD>2$,
set
\beq{eps}
\eps = [\DDD(10+\log \DDD)]^{-1}\,,
\enq
and assume $k > 2/\eps$.

Set $C=600\eps^{-1}$.
With apologies, we now recycle, letting $p = C\frac{\log n}n$,
and take our random graph $G$ to be
the union of
three independent copies, say $G_1,G_2,G_3$, of $\cG(n,p)$.
It is enough to show that $G$ \aas\  contains
the desired paths from $v_1\dots v_m$
(thus giving Theorem~\ref{T3} with $K=3C$).

\medskip
Set $M_0=W_0=\{v_1\dots v_m\}$ and $R=V\sm M_0$.
It is of course enough to show

\begin{claim}\label{claim}
\Aas\ there is an equipartition
$M_1\cup\cdots\cup M_{k-1}$ of $R$
such that
\beq{pm}
\mbox{$G[M_{i-1},M_i]$ admits a perfect matching
for each $i\in [k-1]$\,,}
\enq
\end{claim}

where, for disjoint $A,B\sub V$, $G[A,B]$ is the
bipartite graph on $A\cup B$ with edge set $\nabla_G(A,B)$.

\subsection{Algorithm}

Set
$T=\lfloor mp/6\rfloor$ and $c=mp/T$,
and note that $mp=np/k>C/D>6000$, so $T\ge 1000$.
In what follows we use $N^i(x)$ (respectively\ $N(x)$)
for neighborhood of $x$ in $G_i$
(resp.\ $G$).
We will show (in Section~\ref{Analysis}) that
the following procedure \aas\
produces a partition as in Claim~\ref{claim}.


\bn
{\bf First step:}
Let $\alpha\le 1$ be a constant to be specified later and
$Z=\{x\in R: |N^1(x)\cap W_0|<T\}$.
Let $W_1\dots W_{k-1}$ be disjoint random subsets of $R$
given by
$$
\P(x\in W_i) =\ga /k \left\{\begin{array}{ll}
\forall i\in [k-1] &\mbox{if $x\not\in Z$\,,}\\
\forall i\in \{2\dots k-1\} &\mbox{if $x\in Z$\,,}
\end{array}\right.
$$
these choices independent for different vertices $x$.
(Thus $\P(x\not\in \cup W_i)$ is $1-\ga$ or $1-\alpha(1-1/k)$,
as the case may be.)
Set $W=\cup_{i=0}^{k-1}W_i$.
The $W_i$'s are our initial installments on the $M_i$'s,
to be augmented in the next two steps.
(We won't bother with
formal notation for the evolving $M_i$'s.)

\medskip
It will be helpful to define $L(i)=\{i-1,i+1\}\cap \{0\dots k-1\}$ for
$0\leq i\leq k-1$.
For $i\in [k-1]$, set
\[B_i=\{x\in R\sm W:
\exists j\in L(i), ~|N^1(x)\cap W_j|<T\}\,;\]
these vertices will be barred from $M_i$.
(In particular $B_1\supseteq Z$.)

\mn
{\bf Repair phase}:
For $i\in \{0\dots k-1\}$ and $j\in L(i)$,
let
\[
X_{ij} =
\{x\in W_i:
|N^1(x)\cap W_j|<T\}\,.\]
(In particular $X_{10}=W_1\cap Z=\0$.)
We repair the $X_{ij}$'s in some arbitrary order.
{\em Repairing} $X_{ij}=\{x_1\dots x_s\}$
means that for $\rrr=1\dots s$ we
choose (again, arbitrarily) $T$ available vertices from
$N^2(x_\rrr)$ and add them to $M_j$,
where a vertex is
{\em un}available if it belongs to $B_j$ or
has already been assigned to one of
the $M_u$'s.
Note that the set of edges --- say, $E^2$ --- used in these
``repairs"
(i.e.\ edges from $x_r$ to the chosen vertices in $N^2(x_r)$) is a (star-)forest.

\mn
{\bf Filling in:}
Assign the
as yet unassigned vertices to the $M_i$'s so that
\beq{fill}
\mbox{for all $i$, $~|M_i|=m$ and $M_i\cap B_i=\0$\,.}
\enq

\medskip
The main point in all this is that,
since vertices of $B_i$ are barred from $M_i$ in the
repair and filling in phases, at the end of each of these phases,
we have
$|N_G(x)\cap W_j|\ge T$ for each $x\in M_i$ and
$j\in L(i)$.

\subsection{Analysis}\label{Analysis}

\medskip
We want to show that \aas\  (i) the above
procedure runs to completion and (ii) the $M_i$'s produced satisfy~\eqref{pm}.
(It may be worth observing that
$G_3$, which
plays no role in (i), is needed for (ii).)
Recalling that $\eps$ was specified in \eqref{eps}, set
\[
\mbox{$\ga=1/3$, $~\gc = (1-3\eps)\ga$,
$~\gb = \tfrac{(c\gc-1)^2}{4c\gc}\,,~$ and
$
~q=2e^{-\gb T}\,.$}
\]
\sugg{We will later use the easily verified
\beq{C}
C> \tfrac{c}{\gb}\max\{\tfrac{4}{\eps}+2D\log\tfrac{2e}{\eps},
D\log\tfrac{2C}{\eps Dc}\}\,.
\enq}

We first need some routine observations.

\begin{proposition}\label{devs}
The objects produced by the first step above \aas\  satisfy

\mn
{\rm (a)}  $|Z|\le \eps n$;

\mn
{\rm (b)}  $|W_i|\in (\gc m,(1+\eps)\ga m) ~~\forall i\in [k-1]$;

\mn
{\rm (c)}  $|B_i| <\eps n ~~\forall i\in [k-1]$;

\mn
{\rm (d)}  no vertex is in more than $\eps k$ of the $B_i$'s;

\mn
{\rm (e)}  $|X_{ij}| < 2mq+\log n$ $~\forall i\in \{0\dots k-1\}$
and $j\in L(i)$.
\end{proposition}

\nin
Of course
(c) contains (a), but we state (a) first since it's needed for (b),
which in turn is needed for (c).

Note that the events in Proposition~\ref{devs}
depend only on $G_1$ and the
$W_i$'s.  In fact it will be helpful to conserve some of this information:
for $x\in V$ and $i\in \{0\dots k-1\}$, let
$\gz(i,x)$ be the indicator of the event $\{|N^1(x)\cap W_i|\geq T\}$.
Then $\{x\in Z\}= \{\gz(0,x)=0\}$ ($x\in R$)
and, once we have the $W_i$'s, the remaining assertions in
the proposition are functions of the $\gz(i,x)$'s.

\medskip
There is nothing delicate about Proposition~\ref{devs},
and we aim for
simple rather than optimal arithmetic.
The following
Bernstein/Chernoff-type bound
(for which see e.g. \cite{Beck-Chen}*{Lemma 8.2}) will be sufficient
for our large deviation purposes.  (We use $B(m,\rho)$ for a r.v.
with the binomial distribution $\Bin(m,\rho)$.)
\begin{lemma}\label{Bernstein}
For any m, $\rho$ and $t>0$,
\[\left.
\begin{array}{r}
\P(B(m,\rho)>m\rho+t)\\
\P(B(m,\rho)<m\rho-t)
\end{array}
\right\}<
\exp[-\tfrac{1}{4}\min\{t,
t^2/m\rho \}]\,.\]

\end{lemma}

\begin{proof}[\textbf{\em Proof of Proposition~\ref{devs}}]
(a)
For $x\in R$, we have, using Lemma~\ref{Bernstein},
\begin{eqnarray*}
\P(x\in Z) & = & \P(B( m,p)<T)\\
&=& \P(B(m,p)<mp -(c -1)T)\\
&<& \exp[-\tfrac{(c-1)^2}{4c}T] < q.
\end{eqnarray*}
Thus, writing ``$\succ$" for stochastic domination,
we have $|Z|\prec B(n,q)$,
whence, using Lemma~\ref{Bernstein} and $\eps>2q$,
$\P(|Z| > \eps n) < \exp[-(\eps-q)n/4]$.

\mn
(b)
Given $Z$ satisfying (a) we have, for each $i$,
$|W_i|\sim \Bin(n_i,\ga/k)$, where
$
n_1= n-m-|Z|
$
and $n_i=n-m$ if $i\geq 2$.
In particular (for each $i$),
$n_i\in ((1-2\eps)n,n) $
(note $m<\eps n$ because of our lower bound on $k$),
and
\begin{eqnarray*}
\P\big(|W_i|\not\in (\gc m, (1+\eps) \ga m)\big) &\le &
\P\big(|W_i|\not\in ((1-\epsilon)\alpha n_i/k,(1+\epsilon)\alpha n_i/k)\big) \\
&<& 2\exp[-\eps^2\ga m/4].
\end{eqnarray*}

\nin
(c) and (d).
Condition on values of $Z$ and the $W_i$'s satisfying (a) and (b) --- note
this uses the values $\gz(0,x)$ ($x\in R$) but no other
information from $G_1$ --- and write $\P'$ for the corresponding
conditional probabilities.
(We may of course think of exposing just the edges of $G_1$ incident
with $W_0$ to determine $\P'$.)

For $x\in R\sm W$ and $i\in [k-1]$,
again using Lemma~\ref{Bernstein},
we have
\begin{eqnarray}\label{Bi}
\mbox{$\P'(x\in B_i)$} & < & 2\P(B(\gc m,p)<T)\nonumber\\
&=& 2\P(B(\gc m,p)<\gc mp -(\gc c -1)T) ~< ~q,
\end{eqnarray}
unless $i=1$ and $x\in Z$, in which case $x$ is automatically in $B_1$.
(If $i=1$ and $x\not\in Z$, the 2's in \eqref{Bi} are unnecessary.)

Using \eqref{Bi} and independence of the events $\{x\in B_i\}$
($x\in R\sm W, i\in [k-1]$),
we have
(i)
$|B_1\sm Z|,|B_2|\dots |B_{k-1}|\prec
B(n,q)$, so that (c) holds with probability at least
$1-k\exp[-(\eps-q)n/4]$,
and (ii)  for any $x\in R\sm W$,
\begin{eqnarray*}
\P(|\{i\geq 2:x\in B_i\}|\geq \lceil\eps k\rceil-1 )&<&
\Cc{k}{\lceil\eps k\rceil-1}q^{\eps k-1} \\
&<& (e/\eps)^{\eps k}\exp[- \tfrac{\gb C\eps}{2c}\log n]\\
&<&  \exp[(\DDD \log\tfrac{e}{\eps}-\tfrac{\gb C}{2c})\eps \log n]
=o(1/n).
\end{eqnarray*}
Here we used
$\binom{k}{r}\le(ek/r)^r \le(e /\epsilon  )^{\epsilon k}$,
the latter valid for $r\le \epsilon k$;
$\eps k-1> \eps k/2$; $T=mp/c= C\log n/(ck)$;
$k<\DDD\log n$; and, for the $o(1/n)$,
the easily verified $\gb C/(2c)-\DDD \log(e/\eps)>2/\eps$.)


\mn
(e)
We retain the conditioning and notation $\P'$ of (c).
We assume first that $(i,j)\neq (0,1)$ (and, since $X_{10}=\0$,
may also assume $(i,j)\neq (1,0)$).
For $x\in W_i$ we have, as in \eqref{Bi},
\beq{Xij}
\mbox{$\P'(x\in X_{ij})$} < \P(B(\gc m,p)<T) < q\,,
\enq
whence
$|X_{ij}|\prec B(m,q)$ and (again using Lemma~\ref{Bernstein})
\[\P'(|X_{ij}|\geq 2mq +\log n)<\exp[-(mq+\log n)/4]<n^{-1/4}=o(1/k)\,.\]

For $(i,j)=(0,1)$ the preceding argument is not quite applicable,
since conditioning on
$A:=\{W_1\cap Z=\0\}=\{\gz(0,x)=1 ~\forall x\in W_1\}$
introduces dependencies among the edges joining
$W_0$ and $W_1$.
But since $A$ is an increasing event,
Harris' Inequality \cite{Harris}
says that this conditioning does not increase the
probability of the decreasing event $\{|X_{01}|\geq 2mq+\log n\}$;
so the argument in the preceding paragraph does imply
$\P'(|X_{01}|\geq 2 mq+\log n)=o(1/k)$.
(Of course this detail could also be dealt with by simply
choosing additional random edges between
$W_0$ and $W_1$.)
\end{proof}

\bigskip
Write $Q$ for the intersection of the
events in (a)--(e), and
$S$ for the event that our
process does not get stuck --- that is, there {\em are}
$T$ available vertices whenever the repair phase requires them
and there {\em is} a way to complete the $M_i$'s in the
filling in phase --- and the $M_i$'s it produces satisfy
\eqref{pm}.
We have
\[\P(\ov{S})\leq \P(\ov{Q})+\P(\ov{S}\mid Q) =o(1) +\P(\ov{S}\mid Q)\,,\]
so just need $\P(\ov{S}\mid Q)=o(1)$.

\medskip
The first part of $S$ --- that the process doesn't get stuck --- is easy.
First, given $Q$, the number of
available vertices at any repair step (at $x$ say) is at least
\beq{nmW}
n-m -(|W| + \max_i |B_i|+ T\sum |X_{ij}| ) >n/2\,.
\enq
To see this notice that,
since there are at most $2k$ terms in the sum, we
may bound
the third term in brackets using (e) and
\[Tq=2Te^{-\gb T}\leq 2\tfrac{C}{\DDD c}\exp[-\tfrac{C\gb}{\DDD c}]< \eps
\]
(say).  Here the first inequality is gotten by noting that
$xe^{-\gb x}$ is decreasing on $x>1/\gb$ and that
$T=mp/c \geq C/(\DDD c)$.  The second may be rewritten as
\[ \tfrac{1200}{  c}\exp[-\tfrac{600\gb}{\eps \DDD c}]< \eps^2 D\,,\]
which, since
$c\geq 6$ and $600\gb/c > 5$ (say), follows from the easily verified
\[
200 \exp [-5(10+\log D)] < D^{-1}(10+\log D)^{-2}.
\]
We conclude that the probability that
$x$ has fewer than $T$ available neighbors
in $G_2$ is at most
$
\P(B(n/2,p)<T)
=o(1/n),
$
so that the  repair phase \aas\ finishes successfully.

Second, to say that the filling in step \aas\
finishes successfully, it's enough to show
that $Q$ implies the existence
of an
assignment of $M_i$'s satisfying \eqref{fill}.
This is a standard type of application of Hall's
Theorem, briefly as follows.
For $i\in [k-1]$, write $W_i^*$ for the set of vertices
assigned to $M_i$ through the end of the repair phase,
and set $W^*=\cup W_i^*$, $r_i= m -|W_i^*|$ and
$r=\sum r_i= |R\sm W^*|$.
A set of $M_i$'s with \eqref{fill} is equivalent to a perfect matching
in the bipartite graph $\gG$ on the vertex set
$\{v_{ij}:i\in [k-1], j\in [r_i]\}\cup (R\sm W^*)$
with $v_{ij}\sim x$ iff $x\not\in B_i$.
Then: the common size of the two sides of the bipartition is
$r \in (n/2,n)$ (see \eqref{nmW} for the lower bound);
for degrees in $\gG$
we have
$d(v_{ij}) = |R\sm (W^*\cup B_i)| > n/2>r/2$ (again see \eqref{nmW})
and, using (d),
$
d(x) = r-\sum\{r_i:x \in B_i\} > r-\eps k m >r/2;
$
and it follows easily from Hall's Theorem that a bipartite graph
with $r$ vertices in each part of the bipartition and all
degrees at least $r/2$ admits a perfect matching.

\medskip
We are left with the more interesting part of $S$,
the assertion that \eqref{pm} holds \aas\  given $Q$.
Say $A$ is a {\em violator of type} $(i,j,a)$ if
$A\sub M_i$, $|A|=a$, and $|N_j(A)|<a$,
where $N_j(A)= N(A)\cap M_j$ (and $N(A)=\cup_{x\in A}N(x)$).
By Hall's Theorem it is enough to show the following (given $Q$).

\begin{claim}\label{Claim2}
\Aas\ there is no violator of type $(i,j,a)$ for any
$ a\in \{1\dots \lceil m/2\rceil\}$,
$i\in \{0\dots k-1\}$ and $j\in L(i)$
\end{claim}

\nin
(since if $A$ is a violator of type $(i,j,a)$ for some $a>\lceil m/2\rceil$,
then $M_j\sm N_j(A)$ contains a violator of type $(j,i,\lceil m/2\rceil)$).

\begin{proof}
Fix $i,j$ as in the claim and set $\vt=(ce)^{-2}$.
We consider the cases $a\leq\vt m$ and $a>\vt m$ separately,
beginning with the former.

\mn

Let $E^1$ be the set of edges of $G_1$
that meet $W$, and recall
$E^2$ is the set of edges of $G_2$ that are actually
used in the repair phase.
%
If $A$ is a violator of type
$(i,j,a)$,
then there is some $a$-subset $B$ of $M_j$
containing $N_j(A)$.
(We could, of course, require $|B|<a$.)
The algorithm arranges that
each vertex of $M_i$ is joined by
$E^1\cup E^2$ to at least $T$ vertices of $M_j$
(we actually make no use of $(E(G_1)\sm (E^1))\cup (E(G_2)\sm E^2)$),
whence
\[
|(E^1\cup E^2)\cap\nabla(A,B)| =
|(E^1\cup E^2)\cap\nabla(A,M_j)| \geq aT\,,\]
while (since $E^2$ is a forest)
$|E^2\cap\nabla(A,B)| <2a$; so
\[
|E^1\cap\nabla(A,B)| > a(T-2)\,.
\]
Thus the probability of a violator of type $(i,j,a)$
is at most
\beq{smallbd}
\sum_{A,B}
\P(Q_a(A,B))\,,
\enq
where
$
Q_a(A,B) $ is the event
$\{A\sub M_i,B\sub M_j,|E^1\cap\nabla(A,B)|\geq a(T-2)\}$
if $|A|=|B|=a$
and $Q_a(A,B)=\0$ otherwise,
and the sum is over $A,B\sub V$.
(Of course
if $\{i,j\}\neq\{0,1\}$ then the only nonzero summands are
those with $A,B$ disjoint $a$-subsets of $R$, and, for example,
when $(i,j)=(0,1)$ we are only interested in pairs with $A\sub W_0$
and $B\sub R$ (and $|A|=|B|=a$).)
Note that, summing only over {\em a-subsets} $A,B$ of $V$, we have
\beq{sumAB}
\sum_{A,B} \P(A\sub M_i,B\sub M_j) =
\C{m}{a}^2
\enq
(since the r.v. $\sum_{A,B}{\bf 1}_{\{A\sub M_i,B\sub M_j\}}$ is
actually the constant $\C{m}{a}^2$;
of course by symmetry the summand in \eqref{sumAB}
is the same for all
$(A,B)$ of interest, but we don't need this).

On the other hand, we will show (provided the conditioning event
is not vacuous)
\beq{PrQ}
\P(Q_a(A,B)\mid A\sub M_i,B\sub M_j)
<(1-q)^{-2a}\C{a^2}{a(T-2)}p^{a(T-2)}\,.
\enq
Given this we just need a little arithmetic:
the combination of \eqref{sumAB} and \eqref{PrQ} yields
\begin{eqnarray}\label{sumAB'}
\sum_{A,B}\P(Q_a(A,B))&\leq&
\C{m}{a}^2
(1-q)^{-2a}\C{a^2}{a(T-2)}p^{a(T-2)}\nonumber\\
&\leq &
\left[ (1-q)^{-2}\left(\frac{em}{a}\right)^2
\left(\frac{eap}{T-2}\right)^{T-2}\right]^a\nonumber\\
&=&
\left[ \left(\frac{e}{1-q}\right)^2\left(\frac{a}{m}\right)^{T-4}
\left(\frac{emp}{T-2}\right)^{T-2}\right]^a\nonumber\\
&<&
\left[ (ce)^T\left(\frac{a}{m}\right)^{T-4}\right]^a
\end{eqnarray}

\nin
(say), which easily implies
\beq{smallbd2}
\sum_{a=1}^{\lfloor\vt m\rfloor}\sum_{A,B}\P(Q_a(A,B)) <O(1/m)^{T-4}\,.
\enq

\mn
It remains to prove~\eqref{PrQ}.
Here it is helpful to
think of our procedure as choosing

\mn
(i)
$\gz(0,x)$ for $x\in R$, thus specifying $Z$;

\mn
(ii)  $W_1\dots W_{k-1}$;

\mn
(iii)
$\gz(i,x)$
for $i\in [k-1]$ and $x\in V\sm W=:Y$, thus specifying the $B_i$'s

\mn
(and then continuing).
It is then evident that the only information from
$E(G_1)$ with any bearing on
our choices of the sets $W_l$ and $M_l\sm W_l$ is that in (i) and (iii);
in particular, we have the following.

\begin{observation}\label{obs}
The pair $(M_i,M_j)$, set
$E^1\cap\nabla(W_i\cup W_j, Y)$
and indicators ${\bf 1}_{\{xy\in E^1\}}$ for $(x,y)\in W_i\times W_j$
are conditionally (mutually) independent
given
$W_i,W_j$ and  the values of $\gz(i,x)$ and $\gz(j,x)$
for $x\in Y$.
\end{observation}
\medskip
Suppose now that we're given $W_i,W_j,M_i,M_j$ with $A\sub M_i$
and $B\sub M_j$. For a set $X$ we use $B(X,p)$ for the distribution on
the power set of $X$ that assigns $U\sub X$ probability
$p^{|U|}(1-p)^{|X\sm U|}$.

We assume first that we are not in one of the
slightly special cases with $\{i,j\}=\{0,1\}$.
According to
Observation~\ref{obs}, the sets
$E^1\cap \nabla(x,W_i)$ and $E^1\cap\nabla(x,W_j)$ ($x\in Y$)
and the indicators ${\bf 1}_{\{xy\in E^1\}}$ ($x\in W_i,y\in W_j$)
are mutually independent.
Each of the indicators is Bernoulli with mean $p$;
each
$E^1\cap \nabla(x,W_i)$ is distributed as
$\FF:=B(\nabla(x,W_i),p)$
conditioned on $\{|\FF|\geq T\}$,
an event of probability
at least $\P(\Bin(\gc m,p)\geq T)> 1-q $
(see Proposition~\ref{devs}(b) and \eqref{Bi});
and similarly for the $E^1\cap \nabla(x,W_j)$'s. Thus we can bound the probability of any event determined by
$E^1\cap\nabla(A,B)$, by computing its probability assuming
all edges occur independently with probability $p$, and then multiplying
by $(1-q)^{- |A\setminus W_i|+ |B\setminus W_j|}\le  (1-q)^{-2a}$.
This gives \eqref{PrQ}:
\begin{eqnarray*}
\P(Q_a(A,B)\mid A\sub M_i,B\sub M_j)&<&
(1-q)^{-2a}\P(\Bin(a^2,p)\geq a(T-2))\\
&<& (1-q)^{-2a}\C{a^2}{\lfloor a(T-2)\rfloor}p^{a(T-2)}.
\end{eqnarray*}


When
$\{i,j\}=\{0,1\}$, $Q_a(A,B)$ is determined by the sets
$E^1\cap \nabla(x,W_0)$,
for $x\in A$ if $i=1$ and $x\in B$ if $i=0$.
Recalling that the choice of $M_1$ depends only on $\gz(0,x)$ for $x\in R$, these sets are independent, each distributed as
$\FF:=B(\nabla(x,W_0),p)$
conditioned on $\{|\FF|\geq T\}$,
an event of probability
at least $\P(B(m,p)\geq T)> 1-q $,
and \eqref{PrQ} follows as before.
(In this case $(1-q)^{-2a}$ could be replaced by $(1-q)^{-a}$.)\end{proof}

\bigskip
For the simpler
analysis when $a>\vt m $ (and $a\leq \lceil m/2\rceil$), we just use $G_3$.
Here
a  violator $A$ of type $(i,j,a)$ satisfies
$\nabla(A,B)=\0$ for some $B\sub M_j$ of size $\lceil m/2\rceil$;
so the probability
of such a violator is less than
\[
\sum_{A,B} \P(A\sub M_i,B\sub M_j,
E(G_3)\cap \nabla(A,B)=\0) <4^m (1-p)^{a m/2} = o(1)\,,
\]
where
the sum is over disjoint
$A,B\sub V$
(but really, for example, over
$A,B\sub R$ unless $\{i,j\}=\{0,1\}$)
with $|A|=a$ and $|B|=\lceil m/2\rceil$, and
we used
\[
\sum_{A,B} \P(A\sub M_i,B\sub M_j) =
\C{m}{a}\C{m}{\lceil m/2\rceil}\,,
\]
\[
\P(E(G_3)\cap \nabla(A,B)=\0\mid A\sub M_i,B\sub M_j)\leq
(1-p)^{a m/2}
\]
(of course here $G_3$ is actually independent of the conditioning),
and (recalling $k<D\log n$)
$\vt mp/2 \geq \vt C/(2\DDD)> 2\log 4$.
\qed

\end{document}